\documentclass[a4paper]{amsart}

\usepackage{amsmath, amsthm, amssymb}
\usepackage[english]{babel}
\usepackage[utf8]{inputenc}
\usepackage{mathrsfs}                     
\usepackage{bbm}                          
\usepackage[all]{xy}                      
\usepackage{graphicx}                     
\usepackage{color}
\usepackage{url}
\usepackage{import}                       
\usepackage{todonotes}                    
\usepackage{algpseudocode}                
\usepackage[style=numeric, minnames=4, maxnames=4,
            doi=false, url=false, isbn=false,
            giveninits=true,
            sortcites=true,
            backend=biber]{biblatex}      
\usepackage{enumerate}                    
\usepackage{framed}                       
\usepackage{bm}

\swapnumbers

\theoremstyle{plain}
\newtheorem{theorem}[subsection]{Theorem}
\newtheorem{proposition}[subsection]{Proposition}
\newtheorem{lemma}[subsection]{Lemma}

\newtheorem*{theorem*}{Theorem}
\newtheorem*{proposition*}{Proposition}
\newtheorem*{lemma*}{Lemma}

\theoremstyle{definition}
\newtheorem{definition}[subsection]{Definition}
\newtheorem{remark}[subsection]{Remark}

\newtheorem{openquestion}[subsection]{Open Question}
\newtheorem*{definition*}{Definition}
\newtheorem*{remark*}{Remark}
\newtheorem*{example*}{Example}
\newtheorem*{openquestion*}{Open Question}

\def\al{\alpha}
\def\be{\beta}

\def\ep{\varepsilon}

\def\th{\theta}

\def\la{\lambda}
\def\rh{\varrho}

\def\ph{\varphi}

\def\Ga{\Gamma}

\def\inv{^{-1}}
\def\x{\times}
\def\p{\partial}

\def\R{{\mathbb R}}

\def\Imm{\operatorname{Imm}}

\let\on=\operatorname

\let\mb=\mathbb
\let\mc=\mathcal

\renewcommand{\vec}[1]{\bm{\mathrm{#1}}}

\newcommand{\ud}{\,\mathrm{d}}

\def\todomartins#1{}
\def\todojakob#1{}

\begin{document}

\title[Completeness of Length-Weighted Metrics]{Completeness of Length-Weighted Sobolev Metrics on the Space of Curves}

\author{Martins Bruveris}
\address{Department of Mathematics, Brunel University London,
  Ux\-bridge, UB8 3PH, United Kingdom}
\email{martins.bruveris@brunel.ac.uk}

\author{Jakob M{\o}ller-Andersen}
\address{Department of Mathematics, Florida State University, USA}
\email{jmoeller@math.fsu.edu}
\date{\today}


\begin{abstract}
In this article we prove completeness results for Sobolev metrics with nonconstant coefficients on the space of immersed curves and on the space of unparametrized curves. We provide necessary as well as sufficient conditions for the coefficients of the Riemannian metric for the metric to be metrically complete and we construct examples of incomplete metrics. This work is an extension of previous work on completeness of Sobolev metrics with constant coefficients.
\end{abstract}

\maketitle

\section{Introduction}

Comparison and analysis of geometrical shape has found applications in various fields including image analysis, biomedical imaging and computer vision \cite{Younes2010,Srivastava2016}.

We consider Sobolev metrics on the space $\on{Imm}(S^1,\R^d)$ of closed, regular (or immersed) curves in $\R^d$. Sobolev metrics are metrics of the form
\[
G_c(h,k) = \int_{S^1} a_0 \langle h, k \rangle + 
a_1 \langle D_s h, D_s k \rangle \dots +
a_n \langle D_s^n h, D_s^n k \rangle \ud s \,;
\]
here $c \in \on{Imm}(S^1,\R^d)$ is a curve and $h,k \in T_c \on{Imm}(S^1,\R^d)$ are tangent vectors; $D_sh = h'/|c'|$ and $\ud s = |c'| \ud \th$ denote differentiation and integration with respect arc length respectively. The coefficients $a_k$ can be either constants or functions depending on the curve $c$.

Sobolev metrics on spaces of curves were introduced independently in \cite{Charpiat2007,Mennucci2007,Michor2006c}. Their completeness properties have been studied in~\cite{Mennucci2008,Bruveris2014, Bruveris2015}. Sobolev metrics have been generalized to manifold-valued curves \cite{Celledoni2016, LeBrigant2016_preprint, Su2014} and to higher-dimensional immersed manifolds \cite{Bauer2011b,Bauer2012d}. Numerical discretizations are available for first order metrics~\cite{Michor2008a,Srivastava2011} as well as for second order ones~\cite{Bauer2017,Bauer2015b,Bauer2015c}. See~\cite{Bauer2014} for an overview of Riemannian metrics on spaces of curves and related spaces.

When the coefficients $a_k$ are constants, we call such a metric a Sobolev metric with constant coefficients. It was shown in \cite{Bruveris2014, Bruveris2015} that constant coefficient Sobolev metrics of order $n \geq 2$ are complete: They can be extended to the space $\mc I^n(S^1,\R^d)$ of Sobolev immersions of order $n$ and $\mc I^n(S^1,\R^d)$ equipped with the induced geodesic distance is a complete metric space. Furthermore the geodesic equation has global-in-time solutions and any two curves in the same connected component can be joined by a minimizing geodesic.

However, the results of \cite{Bruveris2014, Bruveris2015} do not cover scale-invariant Sobolev metrics. For $G$ to be scale-invariant we need to choose coefficients $a_k$ that depend on the length $\ell_c$ of the basepoint curve. The choice $a_k(\ell_c) = \ell_c^{2k-3}$ or multiples thereof result in a scale-invariant Riemannian metric $G$. This leads us to consider general length-weighted Sobolev metrics
\[
G_c(h,k) = \int_{S^1} a_0(\ell_c) \langle h, k \rangle + 
a_1(\ell_c) \langle D_s h, D_s k \rangle \dots +
a_n(\ell_c) \langle D_s^n h, D_s^n k \rangle \ud s \,,
\]
with $a_k = a_k(\ell_c)$ smooth functions of $\ell_c$. The purpose of this paper is to extend completeness results to length-weighted Sobolev metrics.

We introduce the following asymptotic conditions on the coefficients
\begin{align}
\max_{1 \leq k \leq n} \int_0^1 r^{1/2-k} \sqrt{a_k(r)} \ud r &= \infty \tag{$I_0$}\,, \\
\max_{1 \leq k \leq n} \int_1^\infty r^{1/2-k} \sqrt{a_k(r)} \ud r &= \infty \tag{$I_\infty$}\,.
\end{align}
These conditions require that at least some coefficient $a_k(\ell)$ does not decay to $0$ too quickly, both for $\ell\to 0$ and for $\ell\to \infty$. By controlling the rate of decay of the coefficients we are able to control the length of curves on metric balls with respect to the geodesic distance.

The main result of this article is that these conditions are sufficient for a Sobolev metric to be complete. We prove the following theorem
\begin{theorem}
Let $G$ be a length-weighted Sobolev metric of order $n \geq 2$, satisfying \eqref{eq:H0} and \eqref{eq:Hinfty}. Then
\begin{enumerate}
\item $(\mc I^n(S^1,\R^d),\on{dist})$ is a complete metric space.
\item $(\mc I^n(S^1,\R^d),G)$ is geodesically complete. 
\item Any two curves in the same connected component can be joined by a minimizing geodesic.
\end{enumerate}
\end{theorem}

The main ingredient in the proof is Proposition~\ref{prop:EquivRieMetricFlatLenWgt}, which shows that $G_c(\cdot,\cdot)$ is equivalent to the background norm $\|\cdot\|_{H^n(d\th)}$ with constants that can be chosen uniformly on arbitrary metric balls with respect to the induced geodesic distance.

In the conditions~\eqref{eq:H0} and~\eqref{eq:Hinfty}, the maximum is taken over $1 \leq k \leq n$. In Section~\ref{sec:IncompletenessCounterEx} we show that this can not be relaxed to $0\leq k \leq n$ by constructing metrically incomplete Sobolev metrics that satisfy the corresponding versions of~\eqref{eq:H0} and~\eqref{eq:Hinfty} with $k=0$.

\section{Background material and notation}

\subsection{The space of curves}

Let $d \geq 1$. The space
\[
\on{Imm}(S^1, \R^d) = \left\{ c \in C^\infty(S^1, \R^d) \,:\, c'(\th) \neq 0 \right\}
\]
of immersions or regular, parametrized curves is an open set in the Fr\'echet space $C^\infty(S^1, \R^d)$ with respect to the $C^\infty$-topology and thus itself a smooth Fr\'echet manifold. For $n \in \mb N$ and $n \geq 2$ the space
\[
\mc I^n(S^1,\R^d) = \left\{ c \in H^n(S^1,\R^d) \,:\, c'(\th) \neq 0 \right\}
\]
of Sobolev curves of order $n$ is similarly an open subset of $H^n(S^1,\R^d)$ and hence a Hilbert manifold. Because of the Sobolev embedding theorem \cite{Adams2003}, $\mc I^n(S^1,\R^d)$ is well-defined and each curve in $\mc I^n(S^1,\R^d)$ is a $C^1$-immersion.

As open subsets of vector spaces the tangent bundles of the spaces $\on{Imm}(S^1,\R^d)$ and $\mc I^n(S^1,\R^d)$ are trivial,
\begin{align*}
T\on{Imm}(S^1,\R^d) &\cong \on{Imm}(S^1,\R^d) \x C^\infty(S^1,\R^d) \\
T\mc I^n(S^1,\R^d) &\cong \mc I^n(S^1,\R^d) \x H^n(S^1,\R^d)\,.
\end{align*}
From a geometric perspective the tangent space at a curve $c$ consists of vector fields along it, i.e., $T_c \on{Imm}(S^1,\R^d) = \Ga(c^\ast T\R^d)$. In the Sobolev case, where $c \in \mc I^n(S^1,\R^d)$, the pullback bundle $c^\ast T\R^d$ is not a $C^\infty$-manifold and the tangent space consists of fibre-preserving $H^n$-maps,
\begin{equation*}
T_c \mc I^n(S^1,\R^2) = 
\left\{h \in H^n(S^1,T\R^d): \quad \begin{aligned}\xymatrix{
& T\R^d \ar[d]^{\pi} \\
S^1 \ar[r]^c \ar[ur]^h & \R^d
} \end{aligned} \right\}\,.
\end{equation*}
See \cite{Michor1997,Hamilton1982} for details in the smooth case and \cite{Eells1966, Palais1968} for spaces of Sobolev maps.

For a curve $c \in \mc I^n(S^1,\R^d)$ or $c \in \on{Imm}(S^1,\R^d)$ we denote the parameter by $\th \in S^1$ and differentiation $\p_\th$ by $'$, i.e., $h' = \p_\th h$. Since $c$ is a $C^1$-immersion, the unit-length tangent vector $D_s c = c'/|c'|$ is well-defined. We will denote by $D_s = \p_\th / |c'|$ the derivative with respect to arc length and by $\ud s = |c'| \ud \th$ the integration with respect to arc length. To summarize, we have
\begin{align*}
D_s & = \frac{1}{|c'|} \p_\th\,, &
\ud s & = |c'| \ud \th\,.
\end{align*}
The length of $c$ is denoted by $\ell_c = \int_{S^1} 1 \ud s$.

\subsection{Sobolev norms}

For $n \geq 1$ we fix the following norm on $H^n(S^1,\R^d)$,
\begin{equation*}
\| h \|_{H^n(d\th)}^2 = \int_{S^1} |h(\th)|^2 + |\p_\th^n h(\th)|^2 \ud \th\,.
\end{equation*}
Its counterpart is the $H^n(ds)$-norm
\begin{equation*}
\| h \|_{H^n(ds)}^2 = \int_{S^1} |h(\th)|^2 + |D_s^{n}h(\th)|^2 \ud s\,,
\end{equation*}
which depends on the curve $c \in \mc I^n(S^1,\R^d)$. The norms $H^n(d\th)$ and $H^n(ds)$ are equivalent, but the constant in the inequalities
\[
C\inv \| h \|_{H^n(d\th)} \leq  \| h \|_{H^n(ds)} \leq C  \| h \|_{H^n(d\th)}
\]
depends on $c$. The nature of this dependence is the content of Proposition~\ref{prop:HndThetaHndSEquiv} and Proposition~\ref{prop:EquivRieMetricFlatLenWgt}.

The $L^2(d\th)$- and $L^2(ds)$-norms are defined similarly,
\begin{align*}
\| u \|^2_{L^2(d\th)} &= \int_{S^1} |u|^2 \ud \th\,, &
\| u \|^2_{L^2(ds)} = \int_{S^1} |u|^2 \ud s\,,
\end{align*}
and they are related via $\left\| u \sqrt{|c'|} \right\|_{L^2(d\th)} = \| u \|_{L^2(ds)}$.

\section{Length-weighted metrics}

\begin{definition}
A \textit{length-weighted Sobolev metric} of order $n$ is a Riemannian metric on $\Imm(S^1,\R^d)$ of the form
\begin{equation}
\label{eq:LengthWgtSobolevMetric}
G_c(h,k) = \sum_{k=0}^n \int_{S^1} a_k(\ell_c) \langle D_s^k h, D_s^k k \rangle \ud s,
\end{equation}
where the coefficients are smooth functions $a_k \in C^\infty(\R_{>0},\R_{\geq 0})$ and $a_0(\ell)>0$ and $a_n(\ell) > 0$ for all $\ell >0$.
\end{definition}

Note that $a_0(\ell),a_n(\ell)>0$, but the coefficients can approach $0$ as $\ell \to 0$ or $\ell \to \infty$. We obtain scale-invariant metrics by choosing $a_k(\ell_c) = b_k \ell_c^{2j-3}$ with $b_k \in \R$. 

\begin{lemma}
\label{lem:smoothMetric}
Let $G$ be a length-weighted Sobolev metric of order $n \geq 0$. Then $G$ is a smooth Riemannian metric on $\on{Imm}(S^1,\R^d)$ and $G$ can be extended to a smooth Riemannian metric on the Sobolev completion $\mc I^k(S^1,\R^d)$ for $k \geq \max(n, 2)$.
\end{lemma}

\begin{proof}
This follows from the smoothness of the maps
\[ D_s : \mc I^n(S^1,\R^d) \x H^{k}(S^1,\R^d) \to H^{k-1}(S^1,\R^d)\,,\quad
(c,h) \mapsto D_s h \,,
\]
for $n \geq 2$ and $1 \leq k \leq n$; see \cite[Lemma~3.3]{Bruveris2014}. Similarly, the map
\[
\mc I^n(S^1,\R^d) \to \R\,,\quad c \mapsto \ell_c\,,
\]
is smooth for $n \geq 2$.
\end{proof}

\subsection{Necessary conditions for completeness}
We are interested in necessary and sufficient conditions on the coefficient functions $a_k$, that would imply metric completeness of the corresponding metric $G$. A necessary condition is that it is neither possible to shrink a curve to a point nor to make it infinitely large along a path of finite length.
Fix $c_0 \in \mc I^n(S^1,\R^d)$ and consider the path $c(t,\th) = \rh(t)c_0(\th)$ with $\rh(0)=1$, $\rh(1) = R$ and $\rh_t(t) > 0$. We want to study the length of the path as $R \nearrow \infty$. The length is
\[
\on{Len}(c) = \int_0^1 \sqrt{G_{\rh c_0}(\rh_t c_0, \rh_t c_0) } \ud t
= \int_1^R \sqrt{G_{r c_0}(c_0, c_0)} \ud r \,. 
\]
Writing $D_c$ for $D_s$ to emphasize the dependence of the operator $D_s$ on the curve,
\[
G_{r c_0}(c_0,c_0) = \int_{S^1} \sum_{k=0}^n a_k(r\ell_{c_0}) 
\left| D_{r c_0}^k c_0 \right|^2 |r c_0'| \ud \th\,.
\]
Assume w.l.o.g. $\ell_{c_0}=1$. Then, since $D_{\rh c_0}^k c_0 = \rh^{-k} D_{c_0}^k c_0$,
\[
G_{rc_0}(c_0, c_0) = \sum_{k=0}^n a_k(r) r^{1-2k} 
\int_{S^1} \left| D_{c_0}^k c_0\right|^2 |c_0'| \ud \th\,.
\]
Note that all integrals in the above sum are strictly positive: $D^{k}_s c_0 \equiv 0$ for some $k > 1$ implies $D^{k-1}_s c_0 \equiv \la$ with $\la \in \R^d$ and $\int_{S^1} D^{k-1}_s c_0 \ud s = 0 $ forces $\la = 0$. Finally $D_s c_0 \equiv 0$ contradicts $c_0$ being an immersion. It follows, that
\begin{align*}
\lim_{R\to \infty} \on{Len}(c) = \infty & \Leftrightarrow
\int_1^\infty \left( \sum_{k=0}^n a_k(r) r^{1-2k} \right)^{1/2} \ud r = \infty \\
& \Leftrightarrow
\sum_{k=0}^n \int_1^\infty r^{1/2-k} \sqrt{a_k(r)} \ud r = \infty \\
& \Leftrightarrow
\int_1^\infty r^{1/2-k} \sqrt{a_k(r)} \ud r = \infty
\text{ for some } 0 \leq k \leq n\,.
\end{align*}
The equivalences are just restatements of the fact that the norms $\| \cdot \|_2, \| \cdot \|_1$ and $\| \cdot \|_\infty$ on $\R^{n+1}$ are equivalent. Thus a necessary condition for completeness is that at least one of the integrals
\[
I_{\infty,k} = \int_1^\infty r^{1/2-k} \sqrt{a_k(r)} \ud r
\]
with $0 \leq k \leq n$ diverges.

Similarly, one can consider the shrinking of a curve to a point by setting $\rh(0) = 1$, $\rh(1) = R > 0$ and $\rh_t(t) < 0$. Then, by a similar argument as above,
\begin{align*}
\lim_{R\to 0} \on{Len}(c) = \infty 
& \Leftrightarrow
\int_0^1 r^{1/2-k} \sqrt{a_k(r)} \ud r = \infty
\text{ for some } 0 \leq k \leq n\,.
\end{align*}
Thus the second necessary condition is the divergence of at least one of the integrals
\[
I_{0,k} = \int_0^1 r^{1/2-k} \sqrt{a_k(r)} \ud r\,,
\]
with $0 \leq k \leq n$. 

\subsection{Sufficient conditions for completeness} The main result of the paper is that for $k \geq 1$ these two conditions are also sufficient for the metric to be complete. We define for length-weighted Sobolev metrics of order $n$  the two properties
\begin{align}
\label{eq:H0}
\max_{1 \leq k \leq n} I_{0,k} &= \infty \tag{$I_0$}\,, \\
\label{eq:Hinfty}
\max_{1 \leq k \leq n} I_{\infty,k} &= \infty \tag{$I_\infty$}\,.
\end{align}
These are sufficient conditions to prevent shrinkage to a point and blow up to infinity of curves along radial paths $c(t,\th) = \rh(t)c_0(\th)$ with finite lengths. We will show that they also prevent finite time shrinkage and blow up along arbitrary paths.

\begin{remark} Note that in~\eqref{eq:H0} and~\eqref{eq:Hinfty} we require $1 \leq k \leq n$. The case when only $I_{0,\infty} = \infty$ or $I_{0,0} = \infty$ is more subtle and in Section \ref{sec:IncompletenessCounterEx} we construct metrics that satisfy $I_{0,0} = \infty$, $I_{0,\infty} = \infty$ but $I_{0,k} < \infty$, $I_{\infty,k} < \infty$ for $1 \leq k \leq n$ and which are not metrically complete. 
\end{remark}

\section{Controlling length and completeness}
In this section we prove that the length $\ell_c$ and the local arc length $|c'(\th)|$ are bounded on geodesic balls. This will constitute the main ingredients for the proof of metric completeness in Theorem~\ref{thm:main}. First we need Poincar\'e-type inequalities for the $L^2(ds)$-norm. Proofs can be found in \cite[Lemma~2.14, 2.15]{Bruveris2014}.
\begin{lemma}
\label{lemma:PoincareIneq}
If $c \in \mc I^2(S^1,\R^d)$ and $h \in H^2(S^1,\R^d)$ then,
\begin{enumerate}
\item 
$\| D_s h\|_{\infty}^2 \leq \dfrac{\ell_c}{4} \| D_s^2 h \|_{L^2(ds)}^2$
\item 
\label{Poincare3} $\| D_s h\|_{L^2(ds)}^2 \leq \dfrac{\ell_c^2}{4} \| D_s^2 h \|_{L^2(ds)}^2$
\end{enumerate}
If $c \in \mc I^n(S^1,\R^d)$ and $h \in H^n(S^1,\R^d)$ then for $0 \leq k \leq n$,
\begin{enumerate}
\setcounter{enumi}{2}
\item 
\label{Poincare4}
$\|D_s^k\|_{L^2(ds)}^2 \leq \| h \|^2_{L^2(ds)} + \| D_s^n h \|^2_{L^2(ds)}$ 
\end{enumerate}
\end{lemma} 

To show that quantities like the length $\ell_c$ or the local arc length $|c'(\th)|$ are bounded we will use the following result, whose proof can be found in~\cite[Lemma~3.2]{Bruveris2015}.

\begin{lemma}
\label{lemma:TangentCondLipschitz}
Let $(M,g)$ be a Riemannian manifold, possibly infinite-dimensional, and $f: M \to F$ be a $C^1$-function into a normed space $F$. Assume that for each each metric ball $B(y,r)$ in $M$ there exists a constant $C$, such that 
\begin{equation}
\| T_x f.v \|_F \leq C(1 + \| f(x) \|_F ) \|v\|_x
\label{eq:TangentCondLipschitz}
\end{equation}
holds for all $x \in B(y,r)$ and $v \in T_x M$. Then the function
\[
f:(M,g) \to (F, \|\cdot \|_F)
\]
is continuous and and Lipschitz continuous on every metric ball. In particular $f$ is bounded on every metric ball. If the constant $C$ is chosen such that \eqref{eq:TangentCondLipschitz} holds globally for $x \in M$, then $f$ is globally Lipschitz continuous.
\end{lemma}

First we show that length $\ell_c$ is bounded on metric balls.

\begin{lemma}
\label{lemma:LengthBound}
Let $G$ be a length-weighted Sobolev metric of order $n \geq 2$ satisfying~\eqref{eq:H0} and \eqref{eq:Hinfty}. Given $c_0 \in \mc I^n(S^1,\R^d)$ and $R > 0$ there exists a constant $C=C(c_0, R)$ such that
\[
C^{-1} \leq \ell_c \leq C\,,
\]
holds for all $c \in \mc I^n(S^1,\R^d)$ with $\on{dist}(c_0, c) < R$.
\end{lemma}
\begin{proof}
The derivative of $\ell_c$ at $c$ in direction $h$ is
\[
D_{c,h} \ell_c = \int_{S^1} \langle D_s h, D_s c \rangle \ud s\,.
\]
Let $1 \leq k \leq n$. We can estimate $D_{c,h} \ell_c$ using Cauchy--Schwartz and Lemma~\ref{lemma:PoincareIneq}~\eqref{Poincare3},
\begin{align*}
|D_{c,h} \ell_c| &\leq \int_{S^1} \left| \langle D_s h, D_sc \rangle \right| |c'| \ud \th
\leq \sqrt{\int_{S^1} |c'| \ud \th} 
\sqrt{ \int_{S^1} \left| \langle D_s h, D_s c \rangle \right|^2 |c'| \ud \th} \\
&\leq \ell_c^{1/2} \| D_s h \|_{L^2(ds)}
\leq 2^{1-k}\ell_c^{k-1/2} \| D_s^k h \|_{L^2(ds)}\,.
\end{align*}
Now define the function
\[
W(r) = \sum_{k=1}^n \int_1^r \rh^{1/2-k} \sqrt{a_k(\rh)} \ud \rh \,.
\]
The assumptions $a_k(\rh) \geq 0$ and $a_n(\rh) > 0$ ensure that $W'(r)> 0$, and \eqref{eq:H0} and \eqref{eq:Hinfty} imply that $\lim_{t \to 0} W(t) = -\infty$ and $\lim_{t \to \infty} W(t) = \infty$. Hence $W:(0,\infty) \to \R$ is a diffeomorphism. We can estimate the derivative $D_{c,h} W(\ell_c)$ via 
\begin{align*}
|D_{c,h} W(\ell_c)| &\leq \sum_{k=1}^n \ell_c^{1/2-k} \sqrt{a_k(\ell_c)}\, 
|D_{c,h} \ell_c| \\
&\leq \sum_{k=1}^n 2^{1-k} \sqrt{a_k(\ell_c})\,  \| D_s^k h \|_{L^2(ds)}
\leq C \sqrt{G_c(h,h)}\,,
\end{align*}
for some constant $C$.
Applying Lemma~\ref{lemma:TangentCondLipschitz} we see that $c \mapsto W \circ \ell_c$ is globally Lipschitz continuous and in particular bounded on every metric ball. Because $W$ is a diffeomorphism, $\ell_c$ itself is also bounded above and away from $0$ on every metric ball.
\end{proof}
We will also need to show that $\log |c'|$ is Lipschitz continuous on metric balls with respect to the geodesic distance.
\begin{lemma}
\label{lemma:logNormCprimeLipCont}
Let $G$ be a length-weighted Sobolev metric of order $n \geq 2$ satisfying~\eqref{eq:H0} and \eqref{eq:Hinfty}. The the function
\begin{align*}
\log |c'| : \left( \mc I^n(S^1,\R^d), \on{dist} \right) \to L^\infty(S^1,\R) \, .
\end{align*}
is continuous and Lipschitz continuous on every metric ball $B(c_0,R)$.

Therefore, there exists a constant $C = C(c_0,R)$ such that all $c \in \mc I^n(S^1,\R^d)$ with $\on{dist}(c_0,c) < R$ and all $\theta \in S^1$ satisfy
\[
C^{-1} \leq |c'(\theta)| \leq C \,.
\]
\end{lemma}

\begin{proof}
Let $c \in \mc I^n(S^1,\R^d)$ and $h \in H^n(S^1,\R^d)$. Then
\begin{equation*}
\left\| D_{c,h} (\log |c'|)\right\|_{L^\infty} = \left\| \langle D_s h, D_s c \rangle \right\|_{L^\infty}
\leq \| D_s h \|_{L^\infty}\,.
\end{equation*}
Using Lemma~\ref{lemma:PoincareIneq} we get
\begin{equation*}
\| D_s h \|_{L^\infty} \leq \frac{\ell_c^{1/2}}{2} \| D^2_s c_t \|_{L^2(ds)} \leq \frac{\ell_c^{n-3/2}}{2^{n-1}} \| D^n_s c_t \|_{L^2(ds)}\,.
\end{equation*}
Fix a metric ball $B(c_0,R)$ and observe that we have
\[
\| D^n_s h \|_{L^2(ds)} \leq \frac{1}{a_n(\ell_c)} \sqrt{G_c(h,h)} \, .
\]
Lemma~\ref{lemma:LengthBound} implies that $\ell_c$ and $\ell_c^{-1}$ are bounded on the ball $B(c_0,R)$ and since $a_n$ is smooth, $a_n(\ell_c)^{-1}$ is bounded as well. Thus
\[
\left\| D_{c,h} (\log |c'|)\right\|_{L^\infty} \leq C \sqrt{ G_c(h,h) }
\]
for some constant $C = C(c_0,R)$ and $c \in B(c_0,R)$. Hence $\log |c'|$ is Lipschitz continuous on every metric ball by Lemma~\ref{lemma:TangentCondLipschitz}. The boundedness of $|c'(\theta)|$ on each metric ball is a direct consequence of Lipschitz continuity.
\end{proof}

Following \cite{Bruveris2014} and \cite[Section~3]{Bruveris2015} we define the following property for a Riemannian metric $G$ on $\Imm(S^1,\R^d)$. It states that a metric satisfying \eqref{assum:Hn} is stronger than a Sobolev metric with constant coefficients of order $n$ with constants that can be chosen uniformly on metric balls.
\begin{itemize}
\item[] Given a metric ball $B(c_0,R)$ in $\Imm(S^1,\R^d)$, there exists a constant $C$, such that
\begin{equation}
\label{assum:Hn}
\| h \|_{H^n(ds)} = \int_{S^1} |h|^n + |D_s h|^n \ud s \leq C G_c(h,h) \tag{$H_n$} 
\end{equation}
holds for all $c \in B(c_0,R)$.
\end{itemize}

The proof of the following proposition can be found in~\cite[Proposition~3.5]{Bruveris2015} and \cite[Lemma~5.1]{Bruveris2014}.

\begin{proposition}
\label{prop:HndThetaHndSEquiv}
Let $G$ be a weak Riemannian metric on $\Imm(S^1,\R^d)$ of order $n \geq 2$ satisfying \eqref{assum:Hn}. Then, given a metric ball $B(c_0,R)$ in $\Imm(S^1,\R^d)$ there exists a constant $C$ such that
\begin{equation*}
C^{-1} \|h \|_{H^n(d\theta)} \leq \|h \|_{H^n(ds)} \leq C \|h\|_{H^n(d\theta)} 
\end{equation*}
holds for all $c \in B(c_0,R)$ and all $h \in H^n(S^1,\R^d)$.
\end{proposition}

We now show that length-weighted Sobolev metrics also have property $\eqref{assum:Hn}$ and are uniformly equivalent to the flat Sobolev metric on any metric ball.
\begin{proposition}
\label{prop:EquivRieMetricFlatLenWgt}
Let $G$ be a length-weighted Sobolev metric of order $n \geq 2$, satisfying \eqref{eq:H0} and \eqref{eq:Hinfty}. Then $G$ satisfies \eqref{assum:Hn}. Furthermore, given a metric ball $B(c_0,R)$ in $\mc I^n(S^1,\R^d)$ there exists a constant $C$ such that
\begin{equation*}
C^{-1} \|h \|_{H^n(d\theta)} \leq \sqrt{G_c(h,h)} \leq C \|h\|_{H^n(d\theta)} 
\end{equation*}
holds for all $c \in B(c_0,R)$ and all $h \in H^n(S^1,\R^d)$.
\end{proposition}

\begin{proof}
Let $B(c_0,R)$ be a geodesic ball in $\mc I^n(S^1,\R^d)$ and denote by $B_{\on{Imm}}(c_0,R)$ the geodesic ball in $\on{Imm}(S^1,\R^d)$. Then \cite[Proposition~A.2]{Bruveris2015} shows that $B_{\on{Imm}}(c_0,R) = B(c_0,R) \cap \on{Imm}(S^1,\R^d)$. Furthermore, $G$ is a strong, smooth Riemannian metric on $\mc I^n(S^1,\R^d)$ and hence geodesic balls are open in the $H^n$-topology. Because $\on{Imm}(S^1,\R^d)$ is dense in $\mc I^n(S^1,\R^d)$, it follows that $B_{\on{Imm}}(c_0,R)$ is dense in $B(c_0,R)$.

By Lemma \ref{lemma:LengthBound} the length $\ell_c$ is bounded on $B(c_0,R)$ above and below, away from $0$. The coefficients $a_k$ are smooth functions and hence $B_k \leq a_k(\ell_c) \leq C_k$, with $B_k,C_k \geq 0$ and $B_0,B_n>0$ for all curves $c \in B(c_0,R)$. With $B = \min(B_0,B_n)$ we have
\[
B \| h \|_{H^n(ds)}^2 \leq \sum_{k = 0}^n \int_{S^1} B_k |D^k_s h|^2 \, \ud s \leq G_c(h,h)\,.
\]
This shows that $G$ satisfies \eqref{assum:Hn}. Using Lemma~ \ref{lemma:PoincareIneq}~\eqref{Poincare4} we get
\begin{equation*}
\sqrt{G_c(h,h)} \leq C' \| h \|_{H^n(ds)}\,,
\end{equation*}
for some constant $C'$. Combining the last two equations with Proposition~\ref{prop:HndThetaHndSEquiv} we obtain the equivalence
\[
C^{-1} \|h \|_{H^n(d\theta)} \leq \sqrt{G_c(h,h)} \leq C \|h\|_{H^n(d\theta)}\,,
\]
for some constant $C$ and all $c \in B_{\on{Imm}}(S^1,\R^d)$. Because $B_{\on{Imm}}(S^1,\R^d)$ is dense in $B(c_0,R)$ and $G$ depends continuously on $c$ the inequalities continue to hold for $c \in B(c_0,R)$.
\end{proof}

\begin{theorem}
\label{thm:main}
Let $G$ be a length-weighted Sobolev metric of order $n \geq 2$, satisfying \eqref{eq:H0} and \eqref{eq:Hinfty}. Then
\begin{enumerate}
\item $(\mc I^n(S^1,\R^d),\on{dist})$ is a complete metric space.
\item $(\mc I^n(S^1,\R^d),G)$ is geodesically complete. 
\item Any two curves in the same connected component can be joined by a minimizing geodesic.
\end{enumerate}
\end{theorem}
\begin{proof}
A verbatim repetition of the proof of \cite[Lemma~4.2]{Bruveris2015} shows the identity map
\[
\on{Id} : (\mc I^n(S^1,\R^d), \on{dist}) \to (\mc I^n(S^1,\R^d), \|\cdot\|_{H^n(d\th)} )
\]
is locally bi-Lipschitz. The remainder of the proof follows \cite[Theorem~4.3]{Bruveris2015}.

Let $(c^j)_{j \in \mathbb{N}}$ be a Cauchy sequence w.r.t the geodesic distance. The sequence is eventually contained in a metric ball $B(c_0,r)$ and hence it is also a Cauchy sequence w.r.t. the $H^n(d\theta)$-norm. Thus there exists a limit $c^* \in H^n(S^1,\R^d)$ and $\|c^j - c^*\|_{H^n(d\theta)} \to 0$. By Lemma \ref{lemma:logNormCprimeLipCont} there exists a constant $C$ such that we have the point-wise lower bound $\| c^j_\theta(\theta)\| \geq C > 0$. This also holds for the limit $c^\ast$ showing that $c^* \in \mc I^n(S^1,\R^d)$. Because the identity map is locally Lipschitz, we also have $\on{dist}(c^j,c^*) \to 0$. Therefore $(\mc I^n(S^1,\R^d),\on{dist})$ is a complete metric space. For a smooth strong Riemannian metric, metric completeness implies geodesic completeness, see \cite[VIII Prop 6.5]{Lang1999}. This is the second part of the statement.

To show the existence of minimizing geodesics we use \cite[Remark~5.4]{Bruveris2015}, which shows the existence of minimizing geodesics for metrics on $\mc I^n(S^1,\R^d)$ that are uniformly bounded and uniformly coercive with respect to the background $\| \cdot \|_{H^n(d\th)}$-norm on metric balls and are of the form
\[
G_c(h,h) = \sum_{k=1}^N \| A_k(c) h \|_{F_k}^2\,,
\]
with some Hilbert spaces $F_k$ and smooth maps $A_k : \mc I^n \to L(H^n,F_k)$, provided the maps $A_k$ have the property
\[
c^j \to c \text{ weakly in }H^1_t\mc I^n_\th
\quad\Rightarrow\quad
A_k(c^j)\dot c^j \to A_k(c) \dot c \text{ weakly in } L^2(I,F_k)\,.
\]
Here $H^1_t\mc I^n_\th = H^1(I, \mc I^n(S^1,\R^d))$. In our case $N=n$, $F_k = L^2(S^1,\R^d)$ and $A_k(c) h = a_k(\ell_c) D_s^k h$. Weak convergence $c^j \to c$ in $H^1_t\mc I^n_\th$ implies convergence $\ell_{c^j} \to \ell_c$ and weak convergence $D_{c^j}^k \dot c^j \to D_{c}^k \dot c$ in $L^2(I,L^2)$ by \cite[Lemma~5.9]{Bruveris2015}. Uniform boundedness and uniform coercivity on metric balls was shown in~\ref{prop:EquivRieMetricFlatLenWgt}. Therefore we obtain the existence of minizing geodesics.
\end{proof}

Let $G$ be a length-weighted Sobolev metric of of order $n \geq 2$, satisfying \eqref{eq:H0} and \eqref{eq:Hinfty}. The same proof as in \cite[Section~4.4]{Bruveris2015} shows that the space $(\on{Imm}(S^1,\R^d), G)$ of smooth immersions is geodesically complete and that its metric completion with respect to the geodesic distance is $\mc I^n(S^1,\R^d)$.

When $c_1, c_2 \in \on{Imm}(S^1,\R^d)$ one can ask if the minimizing geodesic connecting them also lies in $\on{Imm}(S^1,\R^d)$ or only in $\mc I^n(S^1,\R^d)$. The results of~\cite[Section~6]{Bruveris2016c_preprint} generalize naturally to length-weighted metrics, showing that the geodesic is $C^\infty$-smooth provided $c_1$ and $c_2$ are nonconjugate along the geodesic.

\subsection{Shape space}

The metric $G$ is reparametrization invariant, i.e., invariant under the action of $\on{Diff}(S^1)$, and it induces a smooth Riemannian metric on the shape space of unparametrized curves. For technical reasons we have to restrict ourselves to the set $\on{Imm}_{f}(S^1,\R^d)$ of free immersions, meaning those immersions upon which $\on{Diff}(S^1)$ acts freely:
\[
c \in \on{Imm}_{f}(S^1,\R^d) \;\text{ iff }\; \big(c \circ \ph = c \;\Rightarrow\; \ph = \on{Id}_{S^1} \big)\,.
\]
The set $\on{Imm}_{f}(S^1,\R^d)$ is the open and dense set of regular points for the $\on{Diff}(S^1)$-action and we denote the quotient space of unparametrized curves by
\[
B_{i,f}(S^1,\R^d) = \on{Imm}_{f}(S^1,\R^d)/\on{Diff}(S^1)\,.
\]
It is shown in \cite[Section~1.5]{Michor1991} that $B_{i,f}$ is a smooth Fr\'echet manifold and the projection $\pi : \on{Imm}_f \to B_{i,f}$ is a smooth prinicipal fibration with structure group $\on{Diff}(S^1)$. The space of unparametrized Sobolev curves is
\[
\mc B^n(S^1,\R^d) = \mc I^n(S^1,\R^d)/ \mc D^n(S^1)\,,
\]
and we write again $\pi : \mc I^n(S^1,\R^d) \to \mc B^n(S^1,\R^d)$ for the projection.

Following the arguments of \cite[Section~6]{Bruveris2015} one sees that $(\mc B^n(S^1,\R^d), \on{dist}_{\mc B})$ with the quotient metric induced by the geodesic distance is a complete metric space and given $C_1, C_2 \in \mc B^n(S^1,\R^d)$ in the same connected component, there exist $c_1, c_2 \in \mc I^n(S^1,\R^d)$ with $c_1 \in \pi\inv(C_1)$ and $c_2 \in \pi\inv(C_2)$ such that
\[
\on{dist}_{\mc B}(C_1, C_2) = \on{dist}_{\mc I}(c_1, c_2)\,.
\]
Furthermore, $(\mc B^n(S^1,\R^d), \on{dist}_{\mc B})$ is a length space and any two shapes in the same connected component can be joined by a minimizing geodesic. Here a geodesic is to be understood in the sense of metric spaces, because $\mc B^n(S^1,\R^d)$ does not carry the structure of a smooth manifold.

The space $B_{i,f}(S^1,\R^d)$ is a smooth manifold and $G$ induces a Riemannian metric $G^{b}$ on it such that the projection $\pi : (\on{Imm}_f(S^1,\R^d), G) \to (B_{i,f}(S^1,\R^d), G^B)$ is a Riemannian submersion. Then the metric completion of $(B_{i,f}(S^1,\R^d), \on{dist}_B)$ with respect to the geodesic distance is equal to $\mc B^n(S^1,\R^d)$. The proofs of these statements are the same as in \cite[Section~6]{Bruveris2015}.

\section{Counterexample}
\label{sec:IncompletenessCounterEx}
The completeness results were proven under the assumption that \eqref{eq:H0} and \eqref{eq:Hinfty} hold, i.e. we control the behavior of some coefficient $a_1(\ell_c), \dots, a_n(\ell_c)$ other than $a_0(\ell_c)$. Even though controlling the $L^2$-coefficient guarantees long radial paths, it does not allow us to estimate $\ell_c$ along other paths. Here we construct a family of counterexamples which show that controlling the $L^2$-coefficient is not enough for metric completeness. 

We consider the second order metric
\[
G_c(h,k) = \int_{S^1} a_0(\ell_c) \langle h, k \rangle + a_2(\ell_c) \langle D_s^2 h, D_s^2 k \rangle \ud s\,,
\]
with coefficients $a_0(\ell_c) = \ell_c^q$ and $a_2(\ell_c) = \ell^p$ and $p,q \in \R$. The following observations are straight forward,
\begin{align*}
I_{\infty,0} &= \int_1^\infty r^{1/2+ q/2} \ud r\,,
 & I_{\infty,0} = \infty \quad &\Leftrightarrow \quad q \geq -3\,, \\
I_{0,0} &= \int_0^1 r^{1/2+ q/2} \ud r\,,
 & I_{0,0} = \infty \quad &\Leftrightarrow \quad q \leq -3\,, \\
I_{\infty,2} &= \int_1^\infty r^{-3/2+ p/2} \ud r \,,
& I_{\infty,2} < \infty \quad &\Leftrightarrow \quad p < 1\,, \\
I_{0,2} &= \int_0^1 r^{-3/2+ p/2} \ud r\,,
 & I_{0,2} < \infty \quad &\Leftrightarrow \quad p > 1\,.
\end{align*}

In the following thereom we construct two families of metrically incomplete Riemannian metrics. Both families satisfy $I_{\infty,0} = \infty$ and $I_{0,0} = \infty$. The first family additionally satisfies \eqref{eq:H0}, but not \eqref{eq:Hinfty}, while the second family satisfies \eqref{eq:Hinfty}, but not \eqref{eq:H0}. Hence, we cannot extend the maximum in \eqref{eq:H0} and \eqref{eq:Hinfty} to $0 \leq k \leq n$.

\begin{proposition}
\label{prop:counterexample}
The second order length-weighted Sobolev metric
\[
G_c(h,k)  = \int_{S^1} \ell_c^q \langle h, k \rangle + \ell_c^p \langle D_s^2 h, D_s^2 k \rangle \ud s \, ,
\]
is not metrically complete for
\begin{enumerate}
\item $q = -3$ and $p < 1$ or 
\item $q = -3$ and $p > 1$.
\end{enumerate}

\end{proposition}
\begin{proof}
We will construct Cauchy sequences of curves $(c_n)$ with respect to the geodesic distance, such that the sequence $(\ell_{c_n})$ of lengths approaches infinity for the first set of parameters and converges to $0$ for the second set of parameters. Hence the sequence of curves cannot converge to an element of $\mc I^2(S^1,\R^d)$ and hence the space cannot be metrically complete. 

Consider two geometric sequences $(r_n)_{n \in \mb N}$, $(\la_n)_{n \in \mb N}$ with $r_{n+1} = a r_n$, $\la_{n+1} = b \la_n$. We choose $\la_0 > 2$, $b > 1$ and $\la_0, b \in \mb N$; then $\la_n \geq 2$, $\la_n \nearrow \infty$ and $\la_n \in \mb N$. The choice of $r_0$ and $a$ will depend on $p$ and will be made later. Throughout the proof we will write $f \lesssim g$ to denote $f \leq C g$ with a constant $C$ that may depend on $\ep$, $r_0$, $\la_0$, $a$ and $b$.

Define the sequence of curves
\[
c_n(\th) = r_n \left( 1 + \ep \sin(\la_n \th)\right) \vec{n}\,,
\]
with $\vec{n} = (\cos \th, \sin \th)$ and $0 < \ep < \frac 13$. Set $\vec v = (-\sin \th, \cos \th)$; then $\vec n' = \vec v$ and $\vec v' = -\vec n$. The curve $c_n$ is a circle of radius $r_n$ with $2\lambda_n$ bumps of amplitude $\varepsilon$. 

We want to estimate the geodesic distance $\on{dist}(c_n, c_{n+1})$. To do so, we define the intermediate curve
\[
\tilde c_n(\th) = r_{n+1} \left(1 + \ep \sin (\la_{n} \th) \right) \vec{n}\,.
\]
We will estimate $\on{dist}(c_n, \tilde c_n)$ and $\on{dist}(\tilde c_n, c_{n+1})$ separately using linear paths between the curves. The derivatives of $c_n$ are
\begin{align*}
c_n'(\th) &= r_n \left( 1 + \ep \sin(\la_n \th) \right) \vec v 
+ \ep r_n\la_n \cos(\la_n \th)\, \vec n \\
c''_n(\th) &= 2\ep r_n \la_n \cos(\la_n \th)\, \vec v 
- r_n\left(1 + \ep\left(1 + \la_n^2\right) \sin(\la_n \th) \right) \vec n\,.
\end{align*}
We have the following pointwise estimates,
\begin{align*}
|c_n(\th)| &\leq r_n (1 + \ep) \lesssim r_n \\
|c_n'(\th)| &\leq r_n \left( 1 + \ep + \ep \la_n \right) 
\leq r_n \left( 2 + \la_n \right) \lesssim r_n \la_n \\ 
|c''_n(\th)| &\leq r_n \left(1 + \ep + 2\ep\la_n + \ep \la_n^2\right)
\leq r_n(2 + 2\la_n + \la_n^2) \lesssim r_n \la_n^2\,. 
\end{align*}
For $\tilde c_n$ we have the same estimates
\begin{align*}
|\tilde c_n(\th)| & \lesssim r_n &
|\tilde c_n'(\th)| & \lesssim r_n \la_n &
|\tilde c''_n(\th)| & \lesssim r_n \la_n^2 \,,
\end{align*}
because $r_{n+1} \lesssim r_n$.

To estimate $\on{dist}(c_n, \tilde c_n)$ we define the path
\begin{align*}
c(t,\th) &= (1-t) c_n(\th) + t \tilde c_n(\th) \\
&= (1-t) r_n \left(1 + \ep \sin(\la_n \th) \right) \vec n
+ t r_{n+1} \left( 1 + \ep \sin(\la_{n} \th) \right) \vec n \\
&= \left(r_n + t(r_{n+1} - r_n) \right) 
\left(1 + \ep \sin(\la_n \th) \right) \vec n
\end{align*}
Then
\begin{align*}
c'(t,\th) 
&= \left(r_n + t(r_{n+1} - r_n) \right)
\big[ \ep\la_n \cos (\la_n \th) \vec n + (1 + \ep \sin(\la_n \th)) \vec v \big]
\end{align*}
and because $\langle \vec n, \vec v \rangle = 0$, we have the lower bound
\[
|c'| \geq \left(r_n + t(r_{n+1} - r_n) \right) (1 + \ep \sin(\la_n \th))
\gtrsim (1-\ep) r_{n} \gtrsim r_n \,.
\]
We will also need a slightly sharper lower bound for the length $\ell_{c_n}$. Starting from
\[
|c'(t,\th)| \geq (r_n + t(r_{n+1} - r_n)) \ep \la_n \left|\cos(\la_n \th)\right|
\gtrsim r_n \la_n \left| \cos(\la_n \th)\right|\,,
\]
we obtain by integration, since $\int_0^{2\pi} |\cos \la_n \th| \ud \th = 4$ for $\la_n \in \mb N$, the estimate $\ell_{c} \gtrsim r_n \la_n$.
Thus we have
\[
r_n\la_n  \lesssim \ell_{c} \lesssim r_n \la_n\,.
\]
Next we need to estimate the velocity of the path
\[
\p_t c(t,\th) = \tilde c_n(\th) - c_n(\th)\,.
\]
The simple estimate is
\[
|\p_t c| \lesssim r_n\,,
\]
and therefore
\[
\int_{S^1} \left| \p_t c \right|^2 |c'| \ud \th \lesssim r_n^3 \la_n\,.
\]
We also need to estimate $D_s^2 (\p_t c)$. For this we use the formula
\[
D_s^2 h = \frac{1}{|c'|} \left(\frac{1}{|c'|} h'\right)'
= \frac{1}{|c'|^2}h'' - \frac{1}{|c'|^4}\langle c',c''\rangle h'\,.
\]
Up to constants we obtain
\begin{align*}
\left|D_s^2 (\p_t c)\right|
&\lesssim r_{n}^{-2}\cdot r_n \la_n^2 + r_{n}^{-4} \cdot r_n \la_n \cdot r_n \la_n^2 \cdot r_n \la_n \\
&\lesssim r_n^{-1} \la_n^2 + r_n^{-1} \la_n^4 \lesssim r_n^{-1}\la_n^4 \,.
\end{align*}
Thus
\[
\int_{S^1} \left| D_s^2(\p_t c)\right|^2 |c'| \ud \th \lesssim r_n^{-2}\la_n^8 \cdot r_n \la_n \lesssim r_n^{-1} \la_n^9\,.
\]

\begin{figure}
\centering
	\includegraphics[width=.4\textwidth]{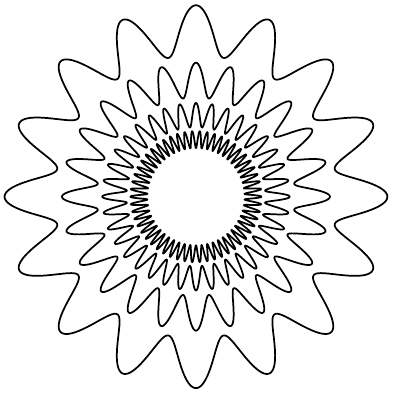}
  \hspace{.1\textwidth}
	\includegraphics[width=.4\textwidth]{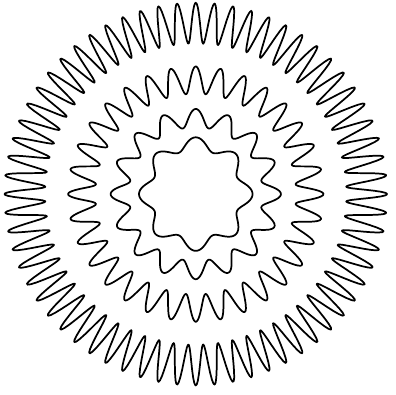}	
\caption{Illustration of the curves $c_n$ using in the counterexample constructed in Proposition~\ref{prop:counterexample}. In both cases $\la_n \to \infty$. The figure on the left shows the case $r_n \to 0$ while the figure on the right shows $r_n \to \infty$.}
	\label{fig:curves}
\end{figure}

We will obtain similar estimates for $\on{dist}(\tilde c_n, c_{n+1})$. Define the path
\begin{align*}
c(t,\th) &= (1-t)\tilde c_n(\th) + t c_{n+1}(\th) \\
&= \big[ r_{n+1} + \ep r_{n+1} \left( (1-t) \sin(\la_{n}\th) 
+ t \sin(\la_{n+1} \th) \right) \!\big] \vec n\,.
\end{align*}
Then
\begin{align*}
c'(t,\th) &= \big[ r_{n+1} + \ep r_{n+1} \left( (1-t) \sin(\la_{n}\th) 
+ t \sin(\la_{n+1} \th) \right) \!\big] \vec v \\
&\qquad {}+ \ep r_{n+1} \left( (1-t) \la_n \cos (\la_n \th) 
+ t \la_{n+1} \cos (\la_{n+1} \th) \right) \vec n\,,
\end{align*}
and
\[
|c'| \geq \big| r_{n+1} + \ep r_{n+1} \left( (1-t) \sin(\la_{n}\th) 
+ t \sin(\la_{n+1} \th) \right) \!\big|
\geq (1 - 2\ep) r_{n+1} \gtrsim r_{n}\,.
\]
We also have the estimate
\[
|c'(t,\th)| \geq \ep r_{n+1}
\left| (1-t) \la_n \cos (\la_n \th) 
+ t \la_{n+1} \cos (\la_{n+1} \th) \right|\,,
\]
which allows us to find a lower bound for the length -- note that $\la_n \in \mb N$,
\begin{align*}
\ell_{c_n} &\gtrsim r_n\la_n \int_0^{2\pi} \left|
(1-t) \cos(\la_n \th) + tb \cos (b \la_n \th) \right| \ud \th \\
&= r_n\la_n \int_0^{2\pi} \left|
(1-t) \cos(\th) + tb \cos (b \th) \right| \ud \th \\
&\gtrsim r_n \la_n\,.
\end{align*}
The last inequality is independent of $t$, because the path $t \mapsto (1-t) \cos \th + tb \cos(b\th)$ into $L^1$ is continuous and does not pass through the zero function. Thus we have again the upper and lower bounds
\[
r_n \la_n \lesssim \ell_c \lesssim r_n \la_n\,,
\]
and we can derive the estimates
\begin{align*}
\int_{S^1} \left| \p_t c\right|^2 |c'| \ud \th & \lesssim r_n^3 \la_n\,, &
\int_{S^1} \left| D_s^2 \p_t c\right|^2 |c'| \ud \th & \lesssim r_n^{-1} \la_n^9\,,
\end{align*}
as before. Now we proceed to choose $r_n$ depending on $p$.

{\bfseries Case (1).}
Note that by the bounds on $\ell_c$, we have
\begin{align*}
a_0(\ell_c) &\lesssim r_n^{-3}\la_n^{-3}\,, &
a_2(\ell_c) &\lesssim r_n^{p}\la_n^{p}\,,
\end{align*}
with $p < 1$. Therefore
\begin{align*}
\on{dist}(c_n, \tilde c_n)^2 &\lesssim r_n^{-3}\la_n^{-3} \cdot r_n^{3}\la_n + 
r_n^{p}\la_n^{p} \cdot r_n^{-1}\la_n^9 \\
& \lesssim \la_n^{-2} + r_n^{p-1}\la_n^{p+9} \,.
\end{align*}
If we choose $r_n = \la_n^{\al}$ for some $\al$ we get the estimate
\[
\on{dist}(c_n, \tilde c_n)^2 \lesssim \la_n^{-2} + \la_n^{\al(p-1) + p+9} \,.
\]
We choose $\al$ such that it satisfies
\[
\al > \frac{p+9}{1-p} > -1\,.
\]
This gives the estimate
\[
\on{dist}(c_n, \tilde c_n)^2 \lesssim \lambda_n^{-2} + \lambda_n^{\be}\,,
\]
with $\be = \al(p-1) + p +9 < 0$, and the same estimates hold for $\on{dist}(\tilde c_n, c_{n+1})$. Therefore 
\[
\on{dist}(c_n, c_{n+1}) \lesssim \la_n^{-1} + \la_n^{\be/2}\,,
\]
and since $\sum_{n} \la_n^{-1} < \infty$ and $\sum_{n} \la_n^{\be/2} < \infty$, it follows that $(c_n)_{n \in \mb N}$ is a Cauchy sequence and 
\[
\ell_{c_n} \gtrsim r_n\la_n = \la_n^{1+\al} \to \infty\,.
\]

{\bfseries Case (2).}
Using the same bounds on $\ell_c$, we have the estimates
\begin{align*}
a_0(\ell_c) &\lesssim r_n^{-3}\la_n^{-3}\,, &
a_2(\ell_c) &\lesssim r_n^{p}\la_n^{p}\,,
\end{align*}
with $p>1$. We choose $r_n = \la_n^\al$ and get the same estimate on the geodesic distance as before,
\[
\on{dist}(c_n, \tilde c_n)^2 \lesssim \la_n^{-2} + \la_n^{\al(p-1) + p+9}\,.
\]
Choosing $\al$ to satisfy
\begin{equation*}
\al < - \frac{p+9}{p-1} < -1\,,
\end{equation*}
 gives
\begin{equation*}
\on{dist}(c_n,\tilde{c}_n)^2 \lesssim \lambda_n^{-2} + \lambda_n^\be\,,
\end{equation*}
with $\be = \al(p-1) + p+9 <0$, and the same estimate holds for $\on{dist}(\tilde c_n, c_{n+1})$. Therefore
\[
\on{dist}(c_n, c_{n+1}) \lesssim \la_n^{-1} + \la_n^{\be/2}\,,
\]
and as before it follows that $(c_n)_{n \in \mb N}$ is a Cauchy sequence and 
\[
\ell_{c_n} \lesssim \la_n^{1+\al} \to 0\,.
\]
\end{proof}

\section{Open questions}

Several open questions remain that we were not able to settle.

\begin{openquestion}
Are the metrics in Proposition~\ref{prop:counterexample} geodesically incomplete?
\end{openquestion}

Let $G$ be one of the metrics from Proposition~\ref{prop:counterexample}. It was shown in Lemma~\ref{lem:smoothMetric} that $G$ is a smooth, strong Riemannian metric on $\mc I^2(S^1,\R^2)$ and thus geodesics with given initial conditions exist for some time. In Proposition~\ref{prop:counterexample} we showed that $G$ is metrically incomplete. In finite dimensions the theorem of Hopf--Rinow states that a Riemannian manifold is metrically complete if and only if it is geodesically complete. In infinite dimensions Atkin~\cite{Atkin1997} has constructed a geodesically complete manifold that is metrically incomplete. Whether the geodesics of $G$ exist for all time remains an open question.

\begin{openquestion}
Are the conditions~\eqref{eq:H0} and~\eqref{eq:Hinfty} necessary for metric completeness?
\end{openquestion}

We showed in Theorem~\ref{thm:main} that the conditions~\eqref{eq:H0} and~\eqref{eq:Hinfty} are sufficient for metric completeness and in Proposition~\ref{prop:counterexample} we constructed some metrically incomplete length-weighted that fail to satisfy these conditions. Whether all length-weighted Sobolev metrics of order $n\geq 2$ that fail both~\eqref{eq:H0} and~\eqref{eq:Hinfty} are metrically incomplete is unknown to us.

\printbibliography

\end{document}